\documentclass{amsart}
\usepackage{graphics}
\usepackage{graphicx}
\usepackage{amsfonts,amsmath,mathrsfs}
\begin{document}

 \newtheorem{thm}{Theorem}[section]
 \newtheorem{cor}[thm]{Corollary}
 \newtheorem{lem}[thm]{Lemma}{\rm}
 \newtheorem{prop}[thm]{Proposition}

 \newtheorem{defn}[thm]{Definition}{\rm}
 \newtheorem{assumption}[thm]{Assumption}
 \newtheorem{rem}[thm]{Remark}
 \newtheorem{ex}{Example}
\numberwithin{equation}{section}
\def\la{\langle}
\def\ra{\rangle}
\def\glexe{\leq_{gl}\,}
\def\glex{<_{gl}\,}
\def\e{{\rm e}}

\def\x{\mathbf{x}}
\def\P{\mathbf{P}}
\def\M{\mathbf{h}}
\def\by{\mathbf{y}}
\def\bz{\mathbf{z}}
\def\F{\mathcal{F}}
\def\R{\mathbb{R}}
\def\H{\mathcal{H}}
\def\T{\mathbf{T}}
\def\N{\mathbb{N}}
\def\D{\mathbf{D}}
\def\V{\mathbf{V}}
\def\U{\mathbf{U}}
\def\K{\mathbf{K}}
\def\H{\mathbf{H}}
\def\Q{\mathbf{Q}}
\def\M{\mathbf{M}}
\def\oM{\overline{\mathbf{M}}}
\def\O{\mathbf{O}}
\def\C{\mathbb{C}}
\def\P{\mathbf{P}}
\def\Z{\mathbb{Z}}
\def\A{\mathbf{A}}
\def\V{\mathbf{V}}
\def\AA{\overline{\mathbf{A}}}
\def\B{\mathbf{B}}
\def\c{\mathbf{C}}
\def\L{\mathcal{L}}
\def\bS{\mathbf{S}}
\def\bnu{\boldsymbol{\nu}}
\def\I{\mathbf{I}}
\def\Y{\mathbf{Y}}
\def\X{\mathbf{X}}
\def\G{\mathbf{G}}
\def\f{\mathbf{f}}
\def\z{\mathbf{z}}
\def\v{\mathbf{v}}
\def\y{\mathbf{y}}
\def\d{\hat{d}}
\def\bx{\mathbf{x}}
\def\bI{\mathbf{I}}
\def\y{\mathbf{y}}
\def\g{\mathbf{g}}
\def\w{\mathbf{w}}
\def\b{\mathbf{b}}
\def\a{\mathbf{a}}
\def\u{\mathbf{u}}
\def\q{\mathbf{q}}
\def\e{\mathbf{e}}
\def\s{\mathcal{S}}
\def\cc{\mathcal{C}}
\def\co{{\rm co}\,}
\def\tg{\tilde{g}}
\def\tx{\tilde{\x}}
\def\tg{\tilde{g}}
\def\tA{\tilde{\A}}
\def\bphi{\boldsymbol{\phi}}
\def\supmu{{\rm supp}\,\mu}
\def\supp{{\rm supp}\,}
\def\cd{\mathcal{C}_d}
\def\cok{\mathcal{C}_{\K}}
\def\cop{COP}
\def\vol{{\rm vol}\,}
\def\om{\mathbf{\Omega}}

\title{Volume of sub-level sets of homogeneous polynomials}
\thanks{Research funded by the European Research Council (ERC) under the European Union's Horizon 2020 research and innovation program (grant agreement ERC-ADG 666981 TAMING).}
\author{Jean B. Lasserre}
\address{LAAS-CNRS and Institute of Mathematics\\
University of Toulouse\\
LAAS, 7 avenue du Colonel Roche\\
31077 Toulouse C\'edex 4, France\\
Tel: +33561336415}
\email{lasserre@laas.fr}

\begin{abstract}
Consider the sub-level set $\K:=\{\x: g(\x)\leq1\}$ of a nonnegative homogeneous polynomial $g$. We show that
its Lebesgue volume ${\rm vol}(\K)$ can be approximated
as closely as desired by solving
a sequence of generalized eigenvalue problems with respect to a pair of Hankel matrices of increasing size, whose entries are obtained 
in closed form. The methodology also extends to compact sets of 
the form $\{\x: a\leq g(\x)\leq b\}$ for non-homogeneous
polynomials with degree $d\ll n$. It reduces the volume computation in $\R^n$ to a ``volume" computation in $\R^d$ (where $d={\rm deg}(g)$) for a certain pushforward measure. Another extension to computing volumes of finite intersections of such sub-level sets
is also briefly described.\\

\end{abstract}
\date{}

\maketitle
\section{Introduction}
Let $g\in\R[\x]_t$ be a nonnegative homogeneous polynomial of degree $t$ (hence $t$ is even) with associated sub-level set
\begin{equation}
\label{set-K}
\K\,:=\,\{\x\in\R^n:\: g(\x)\leq1\,\}.
\end{equation} 
In this paper we describe  an efficient numerical scheme to approximate its Lebesgue volume
${\rm vol}(\K)$ (when finite) as closely as desired. 

\subsection*{Motivation} In addition of being an interesting mathematical problem on its own,
computing ${\rm vol}(\K)$ has also a practical interest outside computational geometry. For instance 
it has a direct link with computing the integral $\int\exp(-g(\x)) d\x$, called 
an {\em integral discriminant} in Dolotin and Morozov \cite{dolotin} and Morozov and Shakirov \cite{morozov}.
Indeed as proved in \cite{morozov}:
\begin{equation}
\label{discriminant}
{\rm vol}(\K)\,=\,\frac{1}{\Gamma(1+\frac{n+t}{2})}\int_{\R^n}\exp(-g(\x))\,d\x,\end{equation}
and to quote \cite{morozov}, {\em ``averaging with exponential weights is an important operation in statistical and quantum physics"}.
However, and again quoting \cite{morozov}, {\em ``despite simply looking, \eqref{discriminant} remains terra incognita"}. However, for special cases of homogeneous polynomials,  the authors in \cite{morozov} have been able to obtain a closed form expression for \eqref{discriminant}
 (hence equivalently for ${\rm vol}(\K)$) in terms of algebraic invariants of $g$.

Various consequences of formula \eqref{discriminant} have been described and exploited in Lasserre \cite{level-homog}.
For instance, ${\rm vol}(\K)$ is a convex function in the coefficients of the polynomial $g$.
In particular this  strong property has been  exploited for proving an extension of the L\"owner-John ellipsoid theorem \cite{lowner} which permits to completely characterize the sublevel set $\K$ (as in \eqref{set-K}) of minimum volume which contains a given set $\om\subset\R^n$ (when minimizing over all positive homogeneous polynomials $g$ of degree $t$). 
However,  computing this sub-level set of minimum volume that contains $\K$ is a computational challenge
since computing (or even approximating) the integral \eqref{discriminant} is a hard problem.

Our main result is that 
\begin{equation}
{\rm vol}(\K)\,=\,\lim_{d\to\infty} \lambda_{\min}(\A_d,\B_d),
\end{equation}
where $\lambda_{\min}(\A_d,\B_d)$ is the smallest generalized eigenvalue of the pair $(\A_d,\B_d)$.
More precisely, ${\rm vol}(\K)$ (and therefore the integral discriminant \eqref{discriminant}) 
is the limit of a monotone sequence of generalized eigenvalue problems with respect to a pair $(\A_d,\B_d)$ of given real Hankel matrices 
of size $d+1$. All entries of both Hankel matrices are easy to obtain in closed-form and the  Hankel matrix $\B_d$ depends 
only on the degree of $g$. Therefore, in principle the integral
\eqref{discriminant} can be approximated efficiently and as closely as desired by (linear algebra) eigenvalue routines.
To the best of our knowledge this result is quite new and in addition, even if we do not provide a closed form expression of \eqref{discriminant}, its
new characterization as a limit or eigenvalue problems could bring new insights. 
Moreover, a first set of numerical experiments on an academic problem
(retrieving the volume of the Euclidean unit ball in $\R^n$) to verify
the behavior of $\lambda_{\min}(\A_d,\B_d)$ as $d$ increases, shows a quick convergence with quite precise approximations 
obtained with relatively small $d$; for instance, with $d=8$,  the relative error is only $0.6\%$ for $n=8$ and $2.15\%$ for $n=9$.

\subsection*{Methodology} Computing (and even approximating) the Lebesgue volume of a convex body is hard (let alone non-convex bodies).
Often the only possibility is to use (non deterministic) Monte Carlo type methods which provide an estimate with statistical guarantees; that is, generate a sample of 
$N$ points according to the uniform distribution on $[-1,1]^n$ and then the ratio 
$\rho_N:=\mbox{(number of points in $\K$)}/N$ provides such an estimate. However $\rho_N$ is a random variable and is neither an upper bound or a lower bound on ${\rm vol}(\K)$. For a discussion on volume computation the interested reader is referred to \cite{sirev} and the many references therein.

However for basic semi-algebraic sets $\K\subset [-1,1]^n$, 
Henrion et al. \cite{sirev} have provided a general methodology to approximate ${\rm vol}(\K)$. 
it consists in solving a hierarchy $(\Q_d)_{d\in\N}$ of semidefinite programs\footnote{A semidefinite program (SDP) is a conic convex optimization problem
with a remarkable modeling power. It can be solved efficiently 
(in time polynomial in its input size) up to arbitrary precision fixed in advance; see e.g. Anjos and Lasserre \cite{handbook}}
of increasing size,
whose associated sequence of optimal values $(\rho_d)_{d\in\N}$ is monotone non increasing 	and converges to ${\rm vol}(\K)$.
An optimal solution of $\Q_d$ is a vector $\y\in\R^{s(2d)}$ (with $s(d)={n+d\choose n}$) whose
each coordinate $y_\alpha$,$\alpha\in\N^n_{2d}$, approximates the $\alpha$-moment of $\lambda_\K$, the restriction to $\K$
of the Lebesgue measure $\lambda$ on $\R^n$ (and therefore $y_0$ approximates ${\rm vol}(\K)$ from above).
An optimal solution of the dual semidefinite program $\Q^*_d$ provides the coefficients $(p_\alpha)_{\alpha\in\N^n_{2d}}$ of a polynomial 
$p\in\R[\x]_{2d}$ which approximates on $[-1,1]$ and from above, the (indicator) function $\x\mapsto 1_\K(\x)=1$ if $\x\in\K$ and $0$ otherwise.
In general the convergence $\rho_d\to{\rm vol}(\K)$ is slow because of a Gibbs phenomenon\footnote{The Gibbs' phenomenon appears at a jump discontinuity when one approximates a piecewise $C^1$ function with a continuous function, e.g. by its Fourier series.} when
one approximates the indicator function $1_\K$ by continuous functions. In \cite{sirev} the authors have proposed a
``trick" which accelerate drastically the convergence but at the price of loosing the monotone convergence $\rho_d\downarrow{\rm vol}(\K)$.
Another acceleration technique was provided in \cite{gauss} which still preserves monotone convergence. It uses the fact that 
moments of $\lambda_\K$ satisfy linear equality constraints that follows from Stokes' theorem. 

Recently, Jasour et al. \cite{jasour} have considered volume computation in the context of risk estimation in uncertain environments.
They have provided an elegant "trick" which reduces computing the $n$-dimensional volume ${\rm vol}(\K)$ to 
computing $\phi([0,1])$ for a certain pushforward measure $\phi$ on the real line,
whose moments are known. (With $\K$ as
in \eqref{set-K} the pushforward measure $\phi$ is with respect to the mapping $g$.)
This results in solving the hierarchy of semidefinite programs proposed 
in \cite{sirev}, but now for measures on the real line as opposed to measures on $\R^n$.
Solving the corresponding hierarchy of dual semidefinite programs
amounts to approximate the indicator of an interval on the real line
by polynomials of increasing degree, and whose coefficients minimize a linear criterion.

On the one hand, it yields drastical computational savings as passing from $\R^n$ to $\R$
is indeed a big and impressive progress. But on the other hand the (monotone) convergence remains slow
as one cannot one cannot apply the acceleration technique based on Stokes' theorem proposed e.g. in \cite{gauss}
because the density of $\phi$ is not known explicitly. In the examples provided in \S \ref{examples} for comparison, we can 
observe this typical (very) slow convergence. However
as the problem is now one-dimensional one may then solve many more steps of the resulting hierarchy of semidefinite programs
provided that one works with a nice basis of polynomials (e.g. Chebyshev polynomials)  to avoid numerical problems as much as possible.
Interestingly, pushforward measures were also used in Magron et al. \cite{projection} to compute the Lebesgue volume of $f(\K)$ for a polynomial mapping $f:\R^n\to\R^m$,
but in this case one has to compute moments of the measure in $\R^n$ whose pushforward measure is the Lebesgue measure on $f(\K)$,
and the resulting computation is still very expensive and limited to modest dimensions.

\subsection*{Contribution}

We provide a simple numerical scheme to approximate ${\rm vol}(\K)$ with $\K$ as in \eqref{set-K} and when
$g$ is positive and homogeneous. To do so we are inspired by the trick of using the pushforward
measure in Jasour et al. \cite{jasour}. The novelty here is that by taking into account the specific nature (homogeneity) of
$g$ in \eqref{set-K} we are able to drastically simplify computations. Indeed, the hierarchy of semidefinite programs
defined in \cite{jasour} can be replaced (and significantly improved) with computing a sequence
of scalars $(\tau_d)_{d\in\N}$ where $\tau_d$ is nothing less than the generalized minimum eigenvalue of two known 
Hankel matrices of size $d$, whose entries are obtained exactly in closed-form with no numerical error.
Therefore there is {\em no} optimization involved anymore.
Moreover, if one uses the basis of orthonormal polynomials w.r.t. the pushforward measure,
then $\tau_d$ is now the minimum eigenvalue of a single real symmetric matrix of size $d$.

At last but not least, the philosophy underlying the methodology also extends 
to arbitrary compact sets of the form $\{\x: a\leq  g(\x)\leq b\}\subset\R^n$ where the polynomial $g$ is not necessarily homogenous and positive. This can be potentially interesting when ${\rm deg}(g)\ll n$ because we reduce the initial Lebesgue volume computation in $\R^n$ to  a $\mu$-volume computation in $\R^d$ (where $d={\rm deg}(g)$) for a certain pushforward measure $\mu$ on $\R^d$ whose sequence of moments is easily obtained in closed form. Moreover, and crucial for the approach, we are still able to 
include additional constraints based on Stokes' theorem, which significantly accelerates the otherwise typically slow convergence. 
In \cite{jasour} the problem would be reduced to a
$\nu$-volume computation only in $\R$ but with no possibility to
include additional Stokes' constraints to accelerate the slow convergence.

\section{Notation and definitions}

\subsection*{Notation}
\label{definitions}
Let $\R[\x]$ denote the ring of polynomials in the variables $\x=(x_1,\ldots,x_n)$
and $\R[\x]_t\subset\R[\x]$ denote the vector space of polynomials  of degree at most $t$, hence of dimension $s(d)={n+t\choose n}$.
Let $\Sigma[\x]\subset\R[\x]$ denote the space of polynomials the are sums-of-squares (in short SOS polynomials) and
let $\Sigma[\x]_d\subset\R[\x]_{2d}$ denote the space of SOS polynomials of degree at most $2d$.
With $\alpha\in\N^n$ and $\x\in\R^n$, the notation $\x^\alpha$ stands for $x_1^{\alpha_1}\cdots x_n^{\alpha_n}$. Also for every $\alpha\in\N^n$, let $\vert\alpha\vert:=\sum_i\alpha_i$
and $\N^n_d:=\{\alpha\in\N^n:\vert\alpha\vert\leq d\}$.

The support of a Borel measure $\mu$ on $\R^n$ is the smallest closed set $\om$ such that $\mu(\R^n\setminus\om)=0$. Denote by $\mathcal{B}(\X)$ the Borel $\sigma$-field associated with a topological space $\X$, and $\mathscr{M}(\X)$ the space of finite Borel measures on $\X$.

Given two real symmetric matrices $\A,\mathbf{C}\in\R^{n\times n}$ denote by 
$\lambda_{\min}(\A,\mathbf{C})$ 
the smallest generalized eigenvalue with respect to the pair $(\A,\mathbf{C})$, that is, the largest scalar $\theta$ such that $\A\x=\theta\,\mathbf{C}\x$ for some vector $\x\in\R^n$.
When $\mathbf{C}$ is the identity matrix then $\lambda_{\min}(\A,\mathbf{C})$ is just the smallest eigenvalue of $\A$. Computing $\lambda_{\min}(\A,\mathbf{C})$ can be done via a pure and efficient linear algebra routine.
The notation $\A\succeq0$ (resp. $\A\succ0$) stands for $\A$ is positive semidefinite (resp. positive definite).

\subsection*{Moment matrix}
Given a real sequence $\bphi=(\phi_\alpha)_{\alpha\in\N^n}$, let
$\M_d(\bphi)$ denote the multivariate (Hankel-type) moment matrix defined by $\M_d(\bphi)(\alpha,\beta)=\phi_{\alpha+\beta}$ for all $\alpha,\beta\in\N^n_d$. 
For instance, in the univariate case $n=1$, with $d=2$, $\M_2$ is the Hankel matrix
\[\M_2(\bphi)\,=\,\left[\begin{array}{ccc} \phi_0 &\phi_1& \phi_2\\
\phi_1 &\phi_2& \phi_3\\
\phi_2 &\phi_3 &\phi_4\end{array}\right].\]
If $\bphi=(\phi_j)_{j\in\N}$ is the moment sequence of a Borel 
measure $\phi$ on $\R$ then $\M_d(\bphi)\succeq0$ for all $d=0,1,\ldots$. Conversely,
if $\M_d(\bphi)\succeq0$ for all $d\in\N$, then $\bphi$ is the moment sequence of some finite Borel measure $\phi$ on $\R$. The converse result is not true anymore in the multivariate case.

Let $\phi,\nu$ be two finite Borel measures on $\R$. The notation $\phi\leq\nu$
stands for $\phi(B)\leq\nu(B)$ for all $B\in\mathcal{B}(\R)$.

\begin{lem}
\label{dominance}
Let $\phi,\nu$ be two finite Borel measures on $\R$ with all moments $\bphi=(\phi_j)_{j\in\N}$
and $\bnu=(\nu_j)_{j\in\N}$ finite. Then $\phi\leq\nu$ if and only if
\[\M_d(\bphi)\,\preceq\,\M_d(\bnu),\quad\forall d=0,1,\ldots\]
\end{lem}
\begin{proof}
{\em Only if part:} $\phi\leq\nu$ implies that $\nu-\phi$ with associated sequence
$\bnu-\bphi=(\nu_j-\phi_j)_{j\in\N}$ is a finite Borel measure on $\R$, and therefore:
\[\M_d(\bnu)-\M_d(\bphi)\,=\,\M_d(\bnu-\bphi)\,\succeq0,\quad d\in\N,\]
i.e., $\M_d(\bnu)\succeq\M_d(\bphi)$ for all $d\in\N$.

{\em If part:} If $\M_d(\bphi)\preceq\M_d(\bnu)$ for all $d\in\N$ then the sequence
$\bnu-\bphi=(\nu_j-\phi_j)_{j\in\N}$ satisfies $\M_d(\bphi-\bnu)\succeq0$ for all $d\in\N$. Therefore,
the moment sequence $\bnu-\bphi=(\nu_j-\phi_j)_{j\in\N}$
of the possibly signed measure $\nu-\phi$ is 
in fact the moment sequence of a finite Borel (positive) measure on $\R$, and therefore
$\nu\geq\phi$.
\end{proof}

\subsection*{Localizing matrix}
Given a real sequence $\bphi=(\phi_\alpha)_{\alpha\in\N^n}$ and a polynomial
$\x\mapsto p(\x):=\sum_\gamma p_\gamma\x^\gamma$, let
$\M_d(p\,\bphi)$ denote the real symmetric matrix defined by: 
\[\M_d(p\,\bphi)(\alpha,\beta)\,=\,\sum_\gamma p_\gamma\,\phi_{\alpha+\beta+\gamma},\quad\alpha,\beta\in\N^n_d.\]
For instance, with $n=1$, $d=2$ and $x\mapsto p(x)=x(1-x)$:
\[\M_2(p\,\bphi)\,=\,\left[\begin{array}{ccc} \phi_1-\phi_2 &\phi_2-\phi_3& \phi_3-\phi_4\\
\phi_2-\phi_3 &\phi_3-\phi_4& \phi_4-\phi_5\\
\phi_3-\phi_4 &\phi_4-\phi_5 &\phi_5-\phi_6\end{array}\right],\]
also a Hankel matrix.
\begin{lem}
\label{putinar}
Let $x\mapsto p(x)=x(1-x)$.

(i) If a real (finite) sequence $\bphi=(\phi_j)_{j\leq 2d}$ satisfies $\M_{d}(\bphi)\succeq0$ 
and $\M_{d-1}(p\,\bphi)\succeq0$, then there is a measure $\mu$ on $[0,1]$
whose moments $\boldsymbol{\mu}=(\mu_j)_{j\leq 2d}$ match $\bphi$.

(ii) If a real (infinite) sequence $\bphi=(\phi_j)_{j\in\N}$ satisfies $\M_{d}(\bphi)\succeq0$ 
and $\M_{d}(p\,\bphi)\succeq0$ for all $d$, then there is a measure $\mu$ on $[0,1]$
whose moments $\boldsymbol{\mu}=(\mu_j)_{j\in\N}$ match $\bphi$.
\end{lem}
See for instance Lasserre \cite{lass-book} and the many references therein.

\subsection*{Pushforward measure} 

Let $\K\subset\R^n$ be a Borel set and $\lambda$ a probability measure on $\K$.
Given a measurable mapping $f:\K\to\R^p$, the pushforward measure 
of $\lambda$ on $\R^p$ w.r.t. $f$ is denoted by $\#\lambda$ and satisfies:
\[\#\lambda(B)\,:=\, \lambda(f^{-1}(B)),\qquad \forall B\in\mathcal{B}(\R^p).\]
In particular, its moments are given by
\begin{equation}
\label{push}
\#\lambda_\alpha\,:=\,\int_{\R^p} \z^\alpha\,\#\lambda(d\z)\,=\,\int_\K f(\x)^\alpha\,\lambda(d\x),\qquad \alpha\in\N^p.
\end{equation}

\subsection*{A version of Stokes' theorem}

Let $\om\subset\R^n$ be an open subset with boundary $\partial\om$ and
let $\x\mapsto \X(\x)$ be a given vector field. Then under suitable smoothness assumptions,
\begin{equation}
\label{stokes}
\int_\om {\rm Div}(X) f(\x)\,d\x+\int_\om \langle \X,\nabla f(\x)\rangle\,d\x\,=\,\int_{\partial\om}\langle \vec{n}_\x,\X\rangle f(\x)\,d\sigma,
\end{equation}
where $\vec{n}_\x$ is the outward pointing normal to $\om$ at $\x\in\partial\om$, and $\sigma$ is the $(n-1)$-dimensional Hausdorff measure on 
the boundary $\partial\om$; see e.g. Taylor \cite[Proposition 3.2, p. 128]{taylor}.

\section{Main result}
\label{homogeneous}
 Let $\B:=[-1,1]^n$ and $\K\subset \B$ be as in \eqref{set-K} with $\partial\K\subset\{\x: g(\x)=1\}$.
 Let  $\lambda$ be the Lebesgue measure on $\B$
normalized to a probability measure so that ${\rm vol}(\K)=2^n\lambda(\K)$. Let $g$ in \eqref{set-K} 
be a nonnegative and homogeneous polynomial of degree $t$. That is, $g(\lambda\x)=\lambda^t g(\x)$ for all $\lambda\in\R$, $\x\in\R^n$.
Denote by:
\begin{equation}
\label{bounds}
\overline{b}_g\,:=\,\max \{\,g(\x): \x\in\B\,\};\quad \underline{a}_g\,:=\,\min \{\,g(\x): \x\in\B\,\}.
\end{equation}
Notice that $\underline{a}_g=0$ because $0\in\B$ and $g$ is nonnegative with $g(0)=0$, and therefore $g(\B)= [0,\overline{b}_g]$.
We next follow an elegant idea of Jasour et al. \cite{jasour}, adapted to the present context.
It reduces the computation 
of ${\rm vol}(\K)$ (in $\R^n$) to a certain volume computation in $\R$, by using a particular pushforward measure of the Lebesgue measure $\lambda$ on $\B$.

Let $\#\lambda$ be the pushforward on the positive half line  of $\lambda$, 
by the 
polynomial mapping $g:\B\to [0,\overline{b}_g]$. From \eqref{bounds}, the support of $\#\lambda$  is the interval $I:=[0,\overline{b}_g]\subset\R$.
Then in view of \eqref{push}:
\begin{equation}
\label{mom}
\#\lambda_k\,:=\,\int_I z^k\,\#\lambda(dz)\,=\,\int_\B g(\x)^k\,\lambda(d\x),\quad k=0,1,\ldots
\end{equation}
All scalars $(\#\lambda_k)_{k\in\N}$ can be obtained in closed form as $g$ is a polynomial and $\lambda$ is the 
(normalized) Lebesgue measure on $\B$. Namely, writing the expansion
\[\x\mapsto g(\x)^k\,=\,\sum_{\alpha\in\N^n_{kd}} g_{k\alpha}\,\x^\alpha,\]
for some coefficients $(g_{k\alpha})$, one obtains:
\begin{equation}
\label{mom-1}
\#\lambda_k\,=\,2^{-n}\,\sum_{\alpha\in\N^n_{kd}} g_{k\alpha}\,\left(\prod_{i=1}^n \frac{(1-(-1)^{\alpha_i+1})}{\alpha_i+1}\right),\quad k=0,1,\ldots
\end{equation}
Next observe that $2^{-n}{\rm vol}(\K)=\#\lambda(g(\K))$ and note that $g(\K)=[0,1]$. Therefore following the recipe introduced in Henrion et al. \cite{sirev}, and with $S:=[0,1]\subset [0,\overline{b}_g]$:

\begin{equation}
\label{def-pb}
\#\lambda(S)\,=\,\max_{\phi\in\mathscr{M}(S)} \{\,\phi(S):\:\phi\,\leq\,\#\lambda\,\}.\end{equation}
Denote by $\phi^*$ the mesure on the real line which is the restriction to
$S\subset I$ of the pushforward measure $\#\lambda$.
That is, 
\begin{equation}
\label{phi-star}
\phi^*(B)\,:=\,\#\lambda(B\cap S), \qquad \forall B\in\mathcal{B}(\R).\end{equation}
Then $\phi^*$ is the unique optimal solution of \eqref{def-pb} and therefore $\phi^*(S)=\#\lambda(S)$
; see e.g. Henrion et al. \cite{sirev}. Then to approximate $\phi^*(S)$ from above, one possibility is to
solve the hierarchy of semidefinite relaxations:
\begin{equation}
\label{relax}
\rho_d\,=\,\displaystyle\max_{\bphi}\{\,\phi_0:\:0\preceq\M_d(\bphi)\,\preceq\,\M_d(\#\lambda);\quad \M_{d-1}(x(1-x)\,\bphi)\succeq0\,\}\end{equation}
where $\bphi=(\phi_j)_{j\leq 2d}$, $\M_d(\#\lambda)$ is the (Hankel) moment matrix (with moments up to order $2d$) associated with the 
pushforward measure $\#\lambda$, and
$\M_d(\bphi)$ (resp. $\M_{d-1}(z(1-z)\,\bphi)$) is the Hankel moment (resp. localizing) matrix (with moments up to order $2d$) 
associated with the sequence $\bphi$ and the polynomial $x\mapsto p(x)=x(1-x)$; see \S \ref{definitions}.
Indeed \eqref{relax} is a relaxation of \eqref{def-pb} and the sequence $(\rho_d)_{d\in\N}$
is monotone non increasing and converges to $\phi^*(S)=\#\lambda(S)$ from above; see e.g. \cite{sirev}.

The dual of \eqref{relax} is the semidefinite program
\begin{equation}
\label{relax-dual}
\begin{array}{rl}
\rho^*_d=\displaystyle\max_{p\in\R[x]_{2d}}\{\,\displaystyle\int p\,d\#\lambda:&p-1=\sigma_0+\sigma_2\,x(1-x)\\
&p,\sigma_0\in\Sigma[x]_d;\quad \sigma_1\in\Sigma[x]_{d-1}\,\},\end{array}
\end{equation}
and if $\K$ has nonempty interior then $\rho^*_d=\rho_d$.

This is the approach advocated by Jasour et al. \cite{jasour} and indeed this reduction 
of the initial (Lebesgue) volume computation in $\R^n$ 
\begin{equation}
\label{multi}
{\rm vol}(\K)\,=\,\max_{\phi\in\mathscr{M}(\K)}\,\{\,\phi(\K):\:\phi\,\leq\,\lambda\,\}
\end{equation}
to instead compute $\#\lambda([0,1])$ (in $\R$) by solving \eqref{def-pb} is quite interesting as it yields drastic computational savings; in fact, solving the multivariate analogues for \eqref{multi} of the 
univariate semidefinite relaxations \eqref{relax} for \eqref{def-pb}, becomes rapidly impossible even for moderate $d$, except for problems of modest dimension (say e.g. $n\leq 4$). 

However it is important to notice that in general the convergence 
$\rho_d\downarrow \#\lambda(S)$ is very slow and numerical problems are expected for large values of $d$. To partially remedy this problem the authors of \cite{jasour} suggest to express moment and localizing matrices in \eqref{relax} in the Chebyshev basis (as opposed to the standard monomial basis).
This allows to solve a larger number of relaxations but it does not change the typical slow 
convergence. The trick based on Stokes' theorem used in \cite{gauss} cannot be used here because 
the dominating (or reference) measure $\#\lambda$ in \eqref{def-pb} is {\em not} the Lebesgue measure $\lambda$ anymore (as in \eqref{multi}). On the other hand 
the trick to accelerate convergence used in \cite{sirev} can still be used, that is,
in \eqref{relax} one now maximizes 
$L_{\bphi}(x(1-x))=\phi_1-\phi_2$ instead of $\phi_0$. If $\bphi^d=(\phi^d_j)_{j\leq 2d}$
is an optimal solution of \eqref{relax} then $\phi^d_0\to\#\lambda(S)$ as $d$ increases but one looses the monotone convergence from above.

In the sequel we show that in the particular case where $g$ is positive and homogeneous then
one can avoid solving the hierarchy \eqref{relax} and instead solve a hierarchy of simple
generalized eigenvalue problems with {\em no} optimization involved and with a much faster convergence.

\subsection{Exploiting homogeneity}

\subsection*{A crucial observation}
Recall that  $2^{-n}{\rm vol}(\K)=\#\lambda(S)$. So let $\phi^*$ be
as in \eqref{phi-star}, and let $\bphi^*=(\phi^*_j)_{j\in\N}$ be its associated sequence of moments.
Consider the vector field $\x\mapsto \X(\x):=\x$. Then ${\rm Div}(\X)=n$.
In addition, as $g$ is homogeneous of degree $t$ then by Euler's identity, $\langle\x,\nabla g(\x)\rangle=t\,g(\x)$ for all $\x\in\R^n$. Recall that $\K\subset{\rm int}(\B)$ and therefore
$g(\x)=1$ for every $\x\in\partial\K$. Next, for every $j\in\N$, as $g(\x)^j=1$ on $\partial\K$ for all $j\in\N$, Stokes' Theorem \eqref{stokes} yields:
\begin{eqnarray*}
0&=&\int_{\partial\K}\langle\vec{n}_\x,\x\rangle\,(1-g(\x)^j)\,d\sigma
\quad\mbox{[as $g(\x)=1$ on $\partial\K$]}\\
&=&n\int_\K (1-g(\x)^j)\,\lambda(d\x)+\int_\K \langle \x,\nabla (1-g(\x)^j)\rangle\,\lambda(d\x)\quad\mbox{[by Stokes]}\\
&=&n\,\lambda(\K)-(n+jt)\,\int_\K g(\x)^j\,\lambda(d\x)\\
&=& n\,\#\lambda(S)-(n+jt)\,\int_{g(\K)}z^j\,\#\lambda(dz)\,=\,n\,\phi^*_0-(n+jt)\,\phi^*_j,
\end{eqnarray*}
so that we have proved:
\begin{lem}
\label{lem-phistar}
Let $\phi^*$ be the Borel measure on $\R$ which is the restriction to $S=[0,1]$ of the pushforward measure $\#\lambda$ on $I$.
Then its moments $\bphi^*=(\phi^*_j)_{j\in\N}$ satisfy :
\begin{equation}
\label{mom-phi}
\phi^*_0\,=\,2^{-n}\,{\rm vol}(\K);\quad \phi^*_j\,:=\,\frac{n}{n+jt}\,\phi^*_0,\quad j=1,2,\ldots
\end{equation}
\end{lem}
Remarkably, Lemma \ref{lem-phistar} relates all moments of $\phi^*$  to its mass $\phi^*_0=2^{-n}{\rm vol}(\K)$ in very simple manner! However it now remains to compute $\phi^*_0$.

\subsection*{Computing $\phi^*_0$}

Define $\M^*_d$ to be the Hankel (moment) matrix with entries:
\begin{equation}
\label{hankel}
\M^*_d(k,\ell)\,:=\,\frac{n}{n+(k+\ell-2)\,t},
\quad k,\ell\,=\,1,2,\ldots,\end{equation}
so that $\phi^*_0\,\M_d^*=\M_d(\bphi^*)$ for all $d\in\N$, where $\M_d(\bphi^*)$ is the Hankel moment matrix associated with the sequence $\bphi^*$ (equivalently with the measure $\phi^*$).

Similarly,  define $\M_{d,x(1-x)}^*$ to be the Hankel matrix with entries:
\begin{equation}
\label{hankel-loc}
\M^*_{d,x(1-x)}(k,\ell)\,:=\,\frac{n}{n+(k+\ell-1)\,t}-\frac{n}{n+(k+\ell-2)\,t},
\quad k,\ell\,=\,1,2,\ldots,\end{equation}
so that $\phi^*_0\,\M_{d,x(1-x)}^*=\M_{d}(x(1-x)\,\bphi^*)$ is the localizing matrix 
associated with $\bphi^*$ and the polynomial $x\mapsto x(1-x)$, for all $d\in\N$. As $\phi^*$ is supported on $[0,1]$ then $\phi^*_0\,\M^*_{d,x(1-x)}\succeq0$ for all $d$, which in turn implies
\begin{equation}
\label{aux}
\M^*_{d,x(1-x)}\,\succeq\,0,\qquad  \forall d\in\N,
\end{equation}
because $\phi^*_0>0$.
\begin{thm}
For each $d\in\N$, let $\M^*_d$ be as in \eqref{hankel} and let $\M_d(\#\lambda)$ be the Hankel 
moment matrix associated with $\#\lambda$ (hence with sequence of moments as in \eqref{mom-1}). Then :
\begin{equation}
\phi^*_0\,=\,\lim_{d\to\infty}\,\lambda_{\min}(\M_d(\#\lambda),\M^*_d),
\end{equation}
i.e., $\phi^*_0$ is the limit of 
a sequence  of minimum generalized eigenvalues associated with the pair $(\M_d(\#\lambda),\M_d^*))$, $d\in\N$.
\end{thm}
\begin{proof}
For every $d\in\N$, let 
\begin{equation}
\label{tau-d}
\tau_d:=\lambda_{\min}(\M_d(\#\lambda),\M^*_d)\,=\,\max\,\{\,\tau:\:\tau\,\M_d^*\,\preceq\,\M_d(\#\lambda)\,\},\end{equation}
which  implies $\tau_d\,\M_d^*\preceq\M_d(\#\lambda)$. In addition, $\tau_d\,\M_{d,x(1-x)}^*\succeq0$
follows from \eqref{aux}. On the other hand, as $\phi^*\leq\#\lambda$, we also have
$\phi^*_0\M^*_d=\M_d(\bphi^*)\preceq\M_d(\#\lambda)$ for all $d\in\N$. Hence,
$\phi^*_0\leq\tau_d$ for all $d\in\N$, and the sequence $(\tau_d)_{d\in\N}$ is monotone non increasing. 

We next show that $\tau_d\downarrow\phi^*_0$ as $d\to\infty$.
As $\tau_d\geq\phi^*_0$ for all $d$,  $\lim_{d\to\infty}\tau_d=\tau^*\geq\phi^*_0$. 
Consider the sequence $\boldsymbol{\mu}=(\mu_j)_{j\in\N}$ defined by:
\[\mu_j\,=\,\tau^*\,\frac{n}{n+jt},\quad j\in\N.\]
Then from $\tau_d\M_d^*\succeq0$ for all $d$, and the convergence $\tau_d\to\tau^*$,
we obtain 
\[\tau^*\,\M^*_d\,=\,\M_d(\boldsymbol{\mu})\,\succeq\,0,\quad \forall d\in\N.\]
For similar reasons 
\[\tau^*\,\M^*_{d,x(1-x)}\,=\,\M_d(x(1-x)\,\boldsymbol{\mu})\,\succeq\,0,\quad \forall d\in\N.\]
By Lemma \ref{putinar},  $\boldsymbol{\mu}$ is the moment sequence of a measure
$\mu$ supported on $[0,1]$ with mass $\mu_0=\mu([0,1])=\tau^*$, and by construction we also have $\mu\leq\#\lambda$. Therefore $\mu\in\mathscr{M}(S)$ is a feasible solution of \eqref{def-pb}
which implies $\mu([0,1])\leq\phi^*_0$. But on the other hand,
\[\phi^*_0\,\leq\,\tau^*\,=\,\mu([0,1])\,\leq\,\phi^*_0,\]
which yields the desired result $\tau^*=\phi^*_0$.
\end{proof}
Therefore to approximate ${\rm vol}(\K)$ from above, one proceeds as follows. Start with $d=1$ and then
\begin{itemize}
\item Compute all moments of $\#\lambda$ up to order $2d$ by \eqref{mom-1}.
\item Compute $\tau_d:=\lambda_{\min}(\M_d(\#\lambda),\M^*_d)$ 
\item set $d=d+1$ and repeat.
\end{itemize}
This produces the required monotone sequence of upper bounds $(\tau_d)_{d\in\N}$ on $\phi^*_0$, which  converges to $\phi^*_0=2^{-n}{\rm vol}(\K)$ as $d$ increases.
Finally,  the following result shows that $\tau_d\leq \rho_d$.
\begin{prop}
For each $d\in\N$, let $\rho_d$ (resp. $\tau_d$) be as in \eqref{relax} (resp. \eqref{tau-d}). Then
$\rho_d\geq\tau_d$.
\end{prop}
\begin{proof}
Consider the sequence $\boldsymbol{\mu}=(\mu_j)_{j\leq 2d}$ defined by:
\[\mu_j\,=\,\tau_d\,\frac{n}{n+jt},\quad j\leq 2d.\]
Then from \eqref{tau-d}, $\tau_d\M_d^*=\M_d(\boldsymbol{\mu})$ and therefore, $0\preceq\M_d(\boldsymbol{\mu})\preceq\M_d(\#\lambda)$. Similarly
$\tau_d\,\M_{d-1,x(1-x)}=\M_{d-1}(x(1-x)\,\boldsymbol{\mu})\succeq0$. In other words,
the sequence $\boldsymbol{\mu}$ is a feasible solution of \eqref{relax}, which implies 
$\mu_0\,(=\tau_d)\,\leq\rho_d$.
\end{proof}
Hence the above eigenvalue procedure (with no optimization involved)
provides a monotone sequence of upper bounds on $\phi^*_0$
that are better than the sequence of upper bounds $(\rho_d)_{d\in\N}$ obtained by solving the hierarchy of semidefinite relaxations \eqref{relax}. 
Notice also that the matrix $\M_d^*$ depends only on the degree of $g$ and not on $g$ itself.

In fact there is a simple interpretation of this improvement. In Problem \eqref{def-pb} and in its
associated semidefinite relaxations \eqref{relax}, one may include the additional constraints 
\begin{equation}
\label{additional}
\phi_j\,=\,n\,\phi_0/(n+jt), \quad\forall j\leq 2d,
\end{equation}
coming from Stokes' theorem applied to $\phi^*$;
see Lemma \ref{lem-phistar}. Indeed we are allowed to do that because $\phi^*$ (which is the unique optimal solution of \eqref{def-pb}) satisfies these additional constraints. If it it does not change the optimal value of \eqref{def-pb} it changes that of \eqref{relax} as it makes the corresponding relaxation stronger and therefore $\tau_d\leq\rho_d$ for all $d$. 

\begin{rem}
\label{rem-1}
{\rm
(i) If one uses the basis of orthonormal polynomials with respect to the pushforward measure $\#\lambda$ then the new moment matrix $\tilde{\M}_d(\#\lambda)$ (expressed in this basis) is the identity matrix and therefore $\tau_d^{-1}$ is the maximum eigenvalue of the corresponding matrix $\tilde{\M}_d^*$ (also expressed in that basis). This basis of orthonormal polynomials can be obtained from the decomposition $\M_d(\#\lambda)=\D\D^T$ for triangular matrices 
$\D$ and $\D^T$. For more details on multivariate orthogonal polynomials, the interested reader is referred to Denkl and Xu \cite{denkl} and the many references therein.

(ii) Alternatively one may also use an orthonormal  basis associated with the sequence of moments $(n/(n+jt)_{j\in\N}$. In this case
the new moment matrix $\tilde{\M}_d^*$ expressed in this basis is the identity matrix and $\tau_d$ is now the minimum eigenvalue
of the new moment $\tilde{\M}_d(\#\lambda)$ expressed in this basis. Notice that the orthonormal basis does not depend on $g$ (only on its degree).

(iii) Finally, another possibility is to simply use the basis of Chebyshev polynomials. Then computing $\tau_d$ (still a generalized eigenvalue problem) involves matrices with much better numerical conditioning.}
 
\end{rem}
\subsection{Some numerical examples}
\label{examples}
To show how this approximation of ${\rm vol}(\K)$ from above by a sequence of 
eigenvalue problems of increasing size is much more efficient than solving
the hierarchy of semidefinite programs \eqref{relax} as suggested in \cite{jasour}, we have considered
a favorable case for \eqref{relax}. We chose $\K$ to be the Euclidean unit ball $\{\x: \Vert\x\Vert\leq 1\}$
with Lebesgue volume $\pi^{n/2}/\Gamma(1+n/2)$ and the ball  $\B$ that contains $\K$ is the smallest
one, i.e., $\B=[-1,1]^n$. Indeed, the smaller is the ball $\B$, the better are the approximation by $\rho_d$ in 
\eqref{relax}.

We first describe the first two steps to appreciate the simplicity of the approach.
Let $n=2$ and $g=\Vert\x\Vert^2=x_1^2+x_2^2$, and $\B=[-1,1]^2$, so that
${\rm vol}(\K)=\pi$. Then:
\[\M_1^*\,=\,\left[\begin{array}{cc}1 & 1/2\\ 1/2 & 1/3\end{array}\right];\quad \M_1(\#\lambda)\,=\,
\left[\begin{array}{cc}1 & 2/3\\ 2/3 & 28/45\end{array}\right]\]
This yields $4\cdot\tau_1\approx3.48$ which is already a  good upper bound on $\pi$ whereas
$4\cdot\rho_1=4$.
\[\M_2^*\,=\,\left[\begin{array}{ccc}1 & 1/2 &1/3\\ 1/2 & 1/3 &1/4\\
1/3 & 1/4 &1/5\end{array}\right];\quad \M_2(\#\lambda)\,=\,
\left[\begin{array}{ccc}1 & 2/3 &28/45\\ 2/3 & 28/45 &24/35\\
28/45 & 24/35 &2/9+8/21+6/25\end{array}\right]\]
This yields $4\cdot\tau_2\approx3.1440$ while 
$4\cdot\rho_2=3.8928$. Hence $4\tau_2$ already provides a very good upper bound on $\pi$
with only moments of order $4$. To appreciate the difference in speed of convergence 
between $\rho_d$ and $\tau_d$, Table \ref{tb:versus} displays both values $\tau_d$ and $\rho_d$  in the case of $n=4$
variables and $d=1,\ldots,5$. While the convergence $\tau_d\to 4.9348$ is quite fast 
with a relative error of $0.03\%$ at step $d=5$, the convergence $\rho_d\to 4.9348$ is extremely 
slow as $\rho_5\approx 8.499$ only; see Figure \ref{fig:compare}

\begin{table}
\begin{center}
\begin{tabular}{|c|c|c|c|c|c|}
\hline
$d$ & $d=1$ & $d=2$ & $d=3$ & $d=4$ & $d=5$ \\
\hline
$\rho_d$& $12.19$ & $11.075$ &   $9.163$ &   $8.878$ &   $ 8.499$\\
\hline
$\tau_d$ & $6.839$ &  $5.309$ &    $5.001$ &    $4.945$ &   $ 4.936$\\
\hline
\end{tabular}
\caption{$n=4$, $\rho^*=4.9348$;  $\rho_d$ versus $\tau_d$  \label{tb:versus}}
\end{center}
\end{table}
\begin{figure}[h!]
\includegraphics[width=0.6\textwidth]{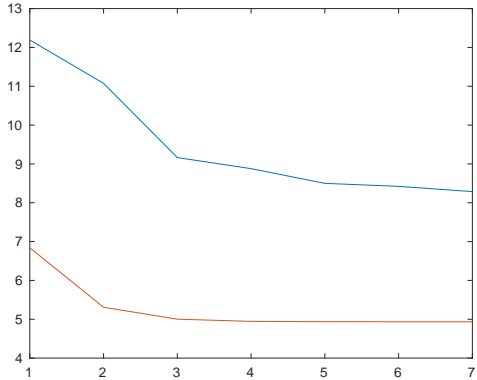}
\caption{$n=4$; Comparing $\tau_d$ (red below) with $\rho_d$ (blue above) \label{fig:compare} } 
\end{figure}

We next provide results for the same problem but now in larger dimensions
$n=5,8,9,10$ in Table \ref{tb:n=5}, Table \ref{tb:n=8}, Table \ref{tb:n=9}, and Table \ref{tb:n=10} respectively.
From inspection we can observe a fast and regular decrease in the value $2^n\tau_d$ as $d$ increases,
and similarly for the relative error.

\begin{table}
\begin{center}
\begin{tabular}{|c|c|c|c|c|c|c|}
\hline
$d$ & $d=1$ & $d=2$ & $d=3$ & $d=4$ & $d=5$ &$d=6$\\
\hline
$\tau_d$& $6.839$ & $5.309$ & $5.001$ & $4.945$ & $4.936$ & $4.935$\\
\hline
$\frac{100(\tau_d-\rho^*)}{\rho^*}$ & $38.6\%$ &$7.58\%$ &$1.35\%$ &$0.22\%$ &$0.03\%$& $0.004\%$ \\
\hline
\end{tabular}
\caption{$n=4$, $\rho^*=4.9348$;  $\tau_d$ and relative error \label{tb:n=4}}
\end{center}
\end{table}

\begin{table}
\begin{center}
\begin{tabular}{|c|c|c|c|c|c|c|}
\hline
$d$ & $d=1$ & $d=2$ & $d=3$ & $d=4$ & $d=5$ &$d=6$  \\
\hline

$2^n\tau_d$ &$10.2892$ & $6.5248$ & $5.57$ & $5.3347$ & $5.2788$ & $5.266$\\
\hline
$\frac{(2^n\tau_d-\rho^*)}{\rho^*}$ & $95\%$ & $23.95\%$ & $5.92\%$ & $1.34\%$ & $0.28\%$ &$0.05\%$ \\
\hline
\end{tabular}
\caption{$n=5$, $\rho^*=5.26$;  $\tau_d$ and relative error \label{tb:n=5}}
\end{center}
\end{table}
\begin{table}
\begin{center}
\begin{tabular}{|c|c|c|c|c|c|c|c|c|}
\hline
$d$ & $d=1$ &$d=2$ & $d=3$ & $d=4$ & $d=5$ &$d=6$ &$d=7$ &$d=8$\\
\hline
$2^n\tau_d$ &   $43.16$ &  $15.04$ &   $ 7.97$ &   $5.569$ &    $4.639$   & $4.272$ &    $4.133$ &   $4.083$\\
\hline
$\frac{(2^n\tau_d-\rho^*)}{\rho^*}$ & $963\%$ & $270\%$  & $96\%$ &   $37\%$ &  $14\%$ &   $5.26\%$ & $1.83\%$ &$ 0.60\%$\\
\hline
\end{tabular}
\caption{$n=8$, 
$\rho^*=4.0587$;  $\tau_d$ and relative error \label{tb:n=8}}
\end{center}
\end{table}
\begin{table}
\begin{center}
\begin{tabular}{|c|c|c|c|c|c|c|c|c|}
\hline
$d$ & $d=1$ &$d=2$ & $d=3$ & $d=4$ & $d=5$ &$d=6$ &$d=7$ &$d=8$\\
\hline
$2^n\tau_d$ &   $73.406$ &  $21.682$ &  $9.801$ &   $ 5.935$ & $4.413$ &  $3.764$ &  $3.485$ &    $3.369$\\
\hline
$\frac{(2^n\tau_d-\rho^*)}{\rho^*}$ & $2125\%$ &   $557\%$ &  $197\%$  & $79\%$ &   $33.8\%$ &   $14.1\%$ &   $5.6\%$  &
 $2.15\%$\\
 \hline
\end{tabular}
\caption{$n=9$, 
$\rho^*=3.298$;  $\tau_d$ and relative error \label{tb:n=9}}
\end{center}
\end{table}

\begin{table}
\begin{center}
\begin{tabular}{|c|c|c|c|c|c|c|}
\hline
$d$ & $d=2$ & $d=3$ & $d=4$ & $d=5$ &$d=6$ &$d=7$ \\
\hline
$2^n\tau_d$ &   $32.432$ &  $12.657$ &    $6.662$ &     $4.375$  &  $3.379$ & $2.921$\\
\hline
$\frac{(\tau_d-\rho^*)}{\rho^*}$ & $1171\%$ & $396.3\%$ & $161\%$ & $71.6\%$ & $32.5\%$ &$14.54\%$\\
\hline
\end{tabular}
\caption{$n=10$, $\rho^*=2.55$;  $\tau_d$ and relative error \label{tb:n=10}}
\end{center}
\end{table}
For $n=10$ and $d=8$, we have encountered numerical problems because the Hankel matrix
$\M_8(\#\lambda)$ is ill-conditioned and then one should use another basis of polynomials
in which to express the matrices $\M^*_8$ and $\M_8(\#\lambda)$; see Remark \ref{rem-1}.

\subsection*{Influence of the size of the ball $\B$}

If one increases the size of the ball $\B=[-r,r]^n$ that contains $\K$ then one expects a slower convergence
and this is why it is recommended to take for $\B$ the smallest ball that contains $\K$. 
An appropriate choice is the box $\prod_{i=1}^n[-u_i,u_i]$ where $u_i$ (resp. $v_i$) is
a lower bound (resp. upper bound) as close as possible to $\min\{\,x_i:\x\in\K\,\}$ 
(resp. $\max\{x_i:\x\in\K\,\}$), which can be computed by the first step of the Moment-SOS hierarchy
described in \cite{lass-book}. From results displayed in Table \ref{tb:influence} with $r=1$ and $r=1.3$, one observes
that even though the convergence is a bit slower it is still quite good. The initial value $\tau_1$ is significantly higher
but then $\tau_d$ (with $r=1.3$) still decreases very fast; see Figure \ref{fig-push}.
\begin{table}
\begin{center}
\begin{tabular}{|c|c|c|c|c|c|c|c|c|}
\hline
$d$ & $d=1$ & $d=2$ & $d=3$ & $d=4$ & $d=5$ &$d=6$ &$d=7$& $d=8$ \\
\hline
$r=1.3;\:(2r)^n\tau_d$ &   $26.345$ &   $11.744$ &   $7.622$ &   $6.149$ &   $5.585$ &    $5.373$ &   $5.299$ & $5.275$ \\
\hline
$r=1;\:(2r)^n\tau_d$ &   $10.289$ &    $6.524$  &  $5.575$ &    $5.334$&     $5.278$& $5.266$ &    $5.264$ &  $5.2639$\\
\hline
\end{tabular}
\caption{$n=5$; ${\rm vol}(\K)=5.2638$; Influence of the size of $\B=[-r,r]^n$ with $r=1$ and $r=1.3$ \label{tb:influence}}
\end{center}
\end{table}
\begin{figure}[h!]
\includegraphics[width=0.6\textwidth]{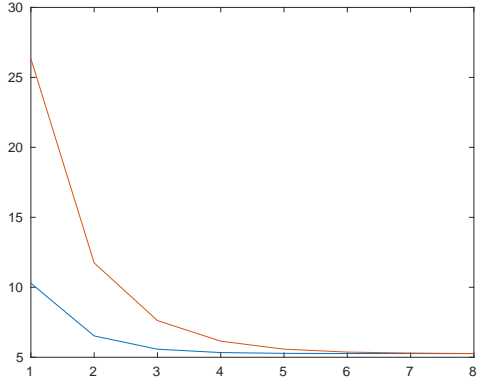}
\caption{$n=8$; Comparing $\tau_d$ with $r=1.3$ (red above) and $r=1$ (blue below)\label{fig-push} } 
\end{figure}

\section{Extensions}
In this section we discuss two extensions of the above methodology,
when:
\begin{itemize}
\item  $\K:=\{\x : a\leq g(\x)\leq b\}\subset(-1,1)^n$ and $g$ is not homogeneous anymore.
\item $\K$ is now $\{\x: g_j(\x)\geq0, \:j=1,\ldots,m\}\subset (-1,1)^n$  and each $g_j$ is homogeneous
(with one of them being nonnegative).
\end{itemize}
In the second extension, again following  Jasour et al. \cite{jasour}, one considers 
the pushforward of the Lebesgue measure by the polynomial mapping $g:\B\to\R^m$
which maps $\x\in\B$ to the vector $(g_j(\x))_{j=1}^m\in\R^m$. Then the initial Lebesgue volume 
computation in $\R^n$ is reduced to an  equivalent ``measure" computation problem
of the form \eqref{multi} but now in $\R^m$ (instead of $\R$) and for which one may apply the 
hierarchy of semidefinite programs described in \cite{sirev}. But as we did in
\S \ref{homogeneous}, we can exploit again the homogeneity of
the $g_j$'s to strengthen the semidefinite relaxations defined in \cite{jasour}, by introducing additional linear constraint 
coming from an appropriate application of Stokes' theorem. The only difference with the univariate case treated in \S \ref{homogeneous}  is that
the problem is not an eigenvalue problem anymore.

The first extension to the non homogeneous case is perhaps more interesting.
We now write $g$ as a sum of homogeneous polynomials of increasing degree $1$, $2,\ldots,{\rm deg}(g)$.
and consider again a pushforward of the Lebesgue mesure $\lambda$ by 
the polynomial mapping $g:\B\to\R^{{\rm deg}(g)}$, $\x \mapsto (g_1(\x),\ldots,g_{{\rm deg}(g)}(\x))$.

\subsection{The non-homogeneous case}

Let $\B=[-1,1]^n$, and suppose that $\K\subset\R^n$ is now described by:
\begin{equation}
\label{set-K-1}
\K\,:=\,\{\,\x: a\,\leq\,g(\x)\,\leq\, b\,\},
\end{equation}
for some $a,b\in\R$, where $g\in\R[\x]_t$, and $\K\subset (-1,1)^n$, possibly after scaling. With no loss of generality we may and will assume that $g(0)=0$ and write
\[\x\mapsto g(\x)\,=\,\sum_{k=1}^t g_k(\x),\quad \x\in\R^n,\]
where for each $1\leq k\leq t$, $g_k$ is a homogeneous polynomial of degree $k$.

We next see how to adapt  the previous methodology of \S \ref{homogeneous} to this more general case in a relatively simple manner. To simplify the exposition and alleviate notation, we will describe the quadratic case $t=2$ but it will become obvious to understand how to
proceed for $t>2$. So with $t=2$,  $g=g_1+g_2$ with 
$g_1$ (resp. $g_2$) homogeneous of degree $1$ (resp. $2$).

Consider the pushforward measure $\#\lambda$ on $\R^2$ of $\lambda$ on $\B$, by the polynomial mapping:
\[g:\B\to\R^2,\quad
\x\mapsto g(\x)=\left[\begin{array}{c}g_1(\x) \\ g_2(\x)\end{array}\right],\quad \x\in\B.\]
Let $\Theta:=g(\B)\subset\R^2$ be the support of the pushforward measure $\#\lambda$, and observe that for each $i,j\in\N$:
\begin{equation}
\label{multi-mom}
\#\lambda_{ij}\,:=\,\int_{\Theta} z_1^iz_2^j\,d\#\lambda(\z)\,=\,
\int_\B g_1(\x)^i \,g_2(\x)^j \,d\lambda(\x),\end{equation}
can be obtained in closed form. Letting
\[S\,:=\,g(\K)\,=\,\{\z\in\Theta: \: a\,\leq\,z_1+z_2\,\leq b\,\},\]
we obtain ${\rm vol}(\K)=\#\lambda(S)$.
Next, recall that (see  \eqref{def-pb} in \S \ref{homogeneous}):
\begin{equation}
\label{multi-vol}
{\rm vol}(\K)\,=\,\#\lambda(S)\,=\,
\displaystyle\max_{\phi\in\mathscr{M}(S)}\,\{\,\phi(S):\:\phi\,\leq\,\#\lambda\,\}
\end{equation}
and $\phi^*$ is the unique optimal solution of \eqref{multi-vol}.

Let $\z\mapsto \tilde{h}(\z):=(b-z_1-z_2)(z_1+z_2-a)$.
The semidefinite relaxations associated with \eqref{multi-vol}
read:
\begin{equation}
\label{relax-multi-1}
\rho_d\,=\,\displaystyle\max_{\boldsymbol{\phi}}\,\{\,\phi_0:\:
0\,\preceq\M_d(\boldsymbol{\phi})\,\preceq\,\M_d(\#\lambda);\quad
\M_d(\tilde{h}\,\boldsymbol{\phi})\,\succeq0\,\},
\end{equation}
where the maximization is over finite bivariate sequences
$\boldsymbol{\phi}=(\phi_{ij})_{i+j\leq 2d}$.

Following the same philosophy as in \S \ref{homogeneous}, we are going to use some additional information on the optimal solution $\phi^*$ of \eqref{multi-vol} to strengthen the semidefinite relaxations 
\eqref{relax-multi-1}. To do so we again use Stokes' theorem.

\subsection*{Stokes} Recall that $\K\subset (-1,1)^n$ and therefore,
$\partial\K\subset\{\x\in\B: h(\x)=0\,\}$ where
$\x\mapsto h(\x):=(b-g_1(\x)-g_2(\x))(g_2(\x)+g_2(\x)-a)$. Therefore by Stokes' theorem,
\begin{eqnarray}
\nonumber
0&=& n\,\int_\K g_1(\x)^i \,g_2(\x)^j\,h(\x)\,d\lambda(\x)\\
\label{multi-stokes}
&&+\int_\K \langle \x,\nabla(g_1(\x)^i \,g_2(\x)^j\,h(\x))\,d\lambda(\x),
\quad\forall i,j\in\N.
\end{eqnarray}
Developing and using homogeneity of $g_1,g_2$, one obtains:
\begin{eqnarray*}
0&=&n\,\displaystyle\int_{S}\left[\,(n+i+2j)\,z_1^iz_2^j\,(b-z_1-z_2)(z_1+z_2-a)\,\right.\\
&&\left.+(z_1+2z_2)\,z_1^iz_2^j\,(a+b-2z_1-2z_2)\,\right].\,d\#\lambda(\z).\end{eqnarray*}
Equivalently, introducing the polynomial $q_{ij}\in\R[\z]$, 
\begin{eqnarray}
\label{poly-q}
\z\mapsto q_{ij}(\z)&:=&(n+i+2j)\,z_1^iz_2^j\,(b-z_1-z_2)(z_1+z_2-a)\\
\nonumber
&&+(z_1+2z_2)\,z_1^iz_2^j\,(a+b-2z_1-2z_2).
\end{eqnarray}
for every $i,j\in\N$, one obtains:
\begin{equation}
\label{multi-stokes-ij}
\int_S q_{ij}(\z)\,d\phi^*(\z)\,=\,0,\quad \forall i,j\in\N.
\end{equation}
Notice that \eqref{multi-stokes-ij}
is  a  linear relation between moments of $\phi^*$, the optimal solution
of \eqref{multi-vol}. That is,
let $\boldsymbol{\phi}^*=(\phi^*_{ij})_{i,j\in\N}$ be the sequence of moments of $\phi^*$ on $S$, and let 
$L_{\boldsymbol{\phi}^*}:\R[\z]\to\R$ be the Riesz functional
\[q\:(=\sum_{i,j}q_{ij}z_1^i\,z_2^j)\quad\mapsto
L_{\boldsymbol{\phi}^*}(q)\,:=\,\sum_{i,j}q_{ij}\,\phi^*_{ij},\quad q\in\R[\z].\]
Then \eqref{multi-stokes-ij} reads
\begin{equation}
\label{stokes-constraints}
L_{\boldsymbol{\phi}^*}(q_{ij})\,=\,0,\quad i,j\in\N,\end{equation}

So we can strengthen the relaxations \eqref{relax-multi-1} by adding the additional (Stokes) moments constraints \eqref{stokes-constraints}, that is, one solves the semidefinite programs:
\begin{equation}
\label{relax-multi-stokes}
\begin{array}{rl}
\tau_d\,=\,\displaystyle\max_{\boldsymbol{\phi}}\,\{\,\phi_0:&
0\,\preceq\M_d(\boldsymbol{\phi})\,\preceq\,\M_d(\#\lambda);
\quad\M_d(\tilde{h}\,\boldsymbol{\phi})\,\succeq0;\\
&L_{\boldsymbol{\phi}^*}(q_{ij})\,=\,0,\quad \mbox{for all $(i,j)$ s.t. ${\rm deg}(q_{ij})\leq 2d$}\,\},
\end{array}
\end{equation}
which is clearly a strengthening of \eqref{relax-multi-1}.
\begin{prop}
\label{prop-multi}
Let $\rho_d$ (resp. $\tau_d$) be as in \eqref{relax-multi-1}
(resp. \eqref{relax-multi-stokes}), $d\in\N$. Then:
\begin{equation}
\label{prop-multi-1}
\#\lambda(S)\,\leq\,\tau_d\,\leq\rho_d\quad\mbox{for all $d$, and
$\tau_d\downarrow\#\lambda(S)$ as $d$ increases.}
\end{equation}
\end{prop}
\begin{proof}
That $\tau_d\leq\rho_d$ for all $d$,  is straightforward and similarly for the monotonicity of the sequence $(\tau_d)_{d\in\N}$. Next,
as $\phi^*$ is the optimal solution of \eqref{multi-vol}, its sequence of
moments $\boldsymbol{\phi}^*=(\phi^*_{ij})$ is feasible for \eqref{relax-multi-stokes}, with associated value $\phi^*_0=\#\lambda(S)={\rm vol}(\K)$. Hence $\tau_d\geq\#\lambda(S)$. Then the convergence $\tau_d\downarrow\#\lambda(S)$ follows from 
$\rho_d\downarrow\#\lambda(S)$.
\end{proof}

The difference with the homogeneous case treated in \S \ref{homogeneous} is that now computing $\tau_d$ requires solving the semidefinite program \eqref{relax-multi-stokes} whereas in \S \ref{homogeneous} computing $\tau_d$ reduces to solving a generalized eigenvalue problem, hence with no optimization involved. However, notice that instead of solving the (costly) $n$-variate semidefinite relaxations associated with \eqref{multi} in $\R^n$, we now solve similar semidefinite relaxations but for a {\em bivariate} problem on the plane. In addition the (convergence) acceleration technique based on 
Stokes's theorem can also be implemented (see \eqref{relax-multi-stokes}).

\subsection{Multi-homogeneous constraints}

Another extension is when $\K=\{\x: g_j(\x)\leq  1,\:j=1,\ldots,m\,\}\subset(-1,1)^n$ for a family $(g_j)_{j=1}^m$ of homogeneous polynomials, 
not necessarily of same degree, say ${\rm deg}(g_j)=t_j$, and at least one of them is positive on $(\R\setminus\{0\})^n$. In this case one may proceed again as suggested in Jasour et al. \cite{jasour}. Now
$\#\lambda$ is the pushforward on $\R^m$ of $\lambda$ on $\B$, by the mapping:
    \[g:\B\to\R^m,\quad g(\x)\,=\,\left[\begin{array}{c} g_1(\x)\\ \cdots\\ g_m(\x)\end{array}\right].\]
   
   In particular $\#\lambda$ has its moments defined by:
   \[\#\lambda_\alpha\,=\,\int_\B g_1(\x)^{\alpha_1}\cdots g_m(\x)^{\alpha_m}\,\lambda(d\x) 
   \,=\,\int_{g(\B)}\z^\alpha\,\#\lambda(d\z),\quad\forall\alpha\in\N^m.\]
   Again all moments $\#\lambda_\alpha$ can be computed in closed form, and again with $S=[0,1]^m$
   $2^{-n}{\rm vol}(\K)=\#\lambda(S)$. 
   Let us describe how the generalization works for the case $m=2$. Again denote by $\phi^*$ on $\R^2$
   the restriction of $\#\lambda$ to $S$ and let $\bphi^*=(\phi^*_{ij})_{i,j\in\N}$ with:
     \[ \phi^*_{ij}\,:=\,\int_{S} z_1^i z_2^j\,\phi^*(d\z),\quad \forall i,j=0,1,\ldots.\]
   So the bivariate analogues of the semidefinite relaxations \eqref{relax} read:
      \begin{equation}
   \label{relax-multi}
\begin{array}{rl}
   \rho_d=\displaystyle\max_{\bphi}\{\,\phi_0:&0\,\preceq\, \M_d(\bphi)\,\preceq \,\M_d(\#\lambda)\\
   &\M_d(x_j(1-x_j)\,\bphi) \succeq0,\quad j=1,2\,\},
   \end{array}
   \end{equation}
   where $\bphi=(\phi_{ij})_{i+j\leq 2d}$, and $\M_d(\bphi)$ (resp. $\M_{d-1}(x_j(1-x_j)\,\bphi)$, $j=1,2$)
   is the moment (resp. localizing) matrix associated with $\bphi$ (resp. with $\bphi$ and $\x\mapsto x_j(1-x_j)$, $j=1,2$). Then $\rho_d\downarrow\#\lambda(S)$ as $d\to\infty$. Again the semidefinite 
   relaxations \eqref{relax-multi} are a lot cheaper to solve than those associated with the $n$-variate problem \eqref{multi}.\\
   
      As we did for the univariate case we can improve the above convergence by adding additional constraints that must be satisfied at the optimal solution $\phi^*$ of \eqref{multi}.
Again $\phi^*_0=2^{-n}{\rm vol}(\K)$.
   Let $(i,j,k,\ell)\in\N^4$ with $k,\ell\geq1$. Then with $X(\x)=\x$,  Stokes's Theorem yields
   
    \begin{eqnarray*}
   0&=&n\,\int_\K g_1^ig_2^j\,(1-g_1)^k(1-g_2)^\ell\lambda(d\x)\\
   &&+\int_\K \langle\x,\nabla [g_1^ig_2^j(1-g_1)^k(1-g_2)^\ell]\rangle\,\lambda(d\x)\\
      &=&n\,\int_{S} z_1^iz_2^j\,(1-z_1)^k (1-z_2)^\ell\,\#\lambda(d\z)\\
      &&+it_1\int_Sz_1^iz_2^j\,(1-z_1)^k (1-z_2)^\ell\,\#\lambda(d\z)\\
      &&+jt_2\int_Sz_1^iz_2^j\,(1-z_1)^k (1-z_2)^\ell\,\#\lambda(d\z)\\
      &&-kt_1\int_Sz_1^{i+1}z_2^j\,(1-z_1)^{k-1} (1-z_2)^\ell\,\#\lambda(d\z)\\
      &&-\ell t_2\int_Sz_1^{i}z_2^{j+1}(1-z_1)^{k} (1-z_2)^{\ell-1}\,\#\lambda(d\z)
\end{eqnarray*}

That is, for each $(i,j,k,\ell)\in\N^4$ with $k,\ell\geq1$ one obtains a linear constraint
                  that links the moments of $\phi^*$, that we denote by $L_{\bphi^*}(q_{ijk\ell})$ where $q_{ijk\ell}\in\R[\z]$ is the above polynomial under the integral sign. For instance,
                  \[0=L_{\bphi^*}(q_{0011})=n(\phi^*_{0}-\phi^*_{10}-\phi^*_{01}+\phi^*_{11})-t_1(\phi^*_0-\phi^*_{01})-t_2(\phi^*_0-\phi^*_{10}).\]
                  \[0=L_{\bphi^*}(q_{1111})=(n+t_1+t_2)\,(\phi^*_{11}+\phi^*_{22}-\phi^*_{21}-\phi^*_{12})
                  -t_1(\phi^*_{21}-\phi^*_{22})-t_2(\phi^*_{12}-\phi^*_{22}),\]
                  etc. So we can add these additional constraints to \eqref{relax-multi} and solve:
 \begin{equation}
 \label{relax-multi-homog}
 \begin{array}{rl}
   \tau_d=\displaystyle\max_{\bphi}\{\,\phi_0:&0\,\preceq\, \M_d(\bphi)\,\preceq \,\M_d(\#\lambda)\\
   &\M_d(x_i(1-x_j)\,\bphi) \succeq0,\quad j=1,2\\\
   &L_{\bphi}(q_{i,j,k,\ell})=0,\quad k,\ell\geq1;\:i+j+k+\ell\leq 2d\,\}.   \end{array}\end{equation}
Of course $\tau_d\leq \rho_d$ for all $d$ and therefore $\tau_d\downarrow\#\lambda(S)$
as $d$ increases.
Again, the difference with the univariate case is that now computing $\tau_d$ requires to solve the semidefinite program \eqref{relax-multi-homog} instead of a generalized eigenvalue problem.
However it is of same dimension as \eqref{relax-multi} and the convergence
$\tau_d\downarrow\#\lambda(S)$ is expected to be much faster than $\rho_d\downarrow\#\lambda(S)$ as we have been able to include additional constraints based on Stokes' theorem.

\section{Conclusion}

We have presented a new methodology to approximate (in principle as closely as desired) the Lebesgue volume of the sub-level set 
$\{\x: g(\x)\leq 1\}$ of a positive homogeneous $n$-variate polynomial $g$. Inspired by 
Jasour et al. \cite{jasour}, we formulate an equivalent ``volume" computation $\mu(I)$ of an interval
$I$ of the real line for a certain pushforward measure $\mu$. The novelty with respect to \cite{jasour} is that by using 
Stokes' theorem and exploiting the homogeneity of $g$, 
we are able to further reduce the problem to a hierarchy of generalized eigenvalue problems for Hankel matrices of increasing size,  with {\em no} optimization involved. To the best of our knowledge, this characterization of Lebesgue volume as the limit of eigenvalue problems of increasing size is new. Moreover, the methodology also extends to sub-level sets of arbitrary polynomials. It then reduces the Lebesgue volume computation in $\R^n$ to a ``volume" computation in $\R^d$
(where $d$ is the degree of the initial polynomial) for a certain pushforward measure. 
An extension to the case several homogeneous constraints has been also presented with a similar pattern 
as in the extension to the single non-homogeneous constraint. 
Preliminary results on a simple case reveal a drastic improvement on the approximation scheme proposed in \cite{jasour}.


\begin{thebibliography}{las}
\bibitem{handbook}
Anjos M. , Lasserre J.B. (Eds.). {\it Handbook of Semidefinite, Conic and Polynomial Optimization}, Springer, New York, 2012.
\bibitem{denkl}
C.F.  Denkl, Y. Xu. {\em Orthogonal Polynomial in Several Variables}, Cambridge University Press, Cambridge, UK, 2014.
\bibitem{dolotin}
V. Dolotin and A. Morozov, {\em Introduction to Non-Linear Algebra}, World Scientific, 2007.
\bibitem{sirev}
D. Henrion, J.B. Lasserre, and C. Savorgnan,  {\em Approximate volume and integration for basic semialgebraic sets},  SIAM Review  
51(2009),  pp. 722--743.
\bibitem{jasour}
A. Jasour, A. Hofmann, and B.C. Williams, {\em Moment-Sum-Of-Squares Approach For Fast Risk 
Estimation In Uncertain Environments}, {\tt arXiv:1810.01577}, 2018.
\bibitem{level-homog}
J.B. Lasserre, {\em Level sets and non Gaussian integrals of positively homogeneous functions},
Int. Game Theory Review 17 (2015), No 1,   28 pages.
\bibitem{lowner}
J.B. Lasserre, {\em A generalization of L\"owner-John's ellipsoid theorem}, Math. Program. 152 (2015), pp. 559--591.
\bibitem{gauss}
J.B. Lasserre, {\em Computing Gaussian \& exponential measures of semi-algebraic sets},  Adv. Appl. Math. 91(2017),  pp. 137--163.
\bibitem{lass-book}
J.B. Lasserre, {\em Moments, Positive Polynomials and Their Applications}, Imperial College Press, 2010.
\bibitem{projection}
V. Magron, D. Henrion, and J.B. Lasserre, {\em Semidefinite approximations of projections and polynomial images of semi-algebraic sets},
SIAM J. Optim. 25 (2015), pp. 2143--2164.
\bibitem{morozov}
A. Morozov and S. Shakirov, {\em Introduction to integral discriminants}, J. High Energy
Phys. 12 (2009), No. 12.
\bibitem{taylor}
M.E. Taylor, {\em Partial Differential Equations: Theory}, Springer-Verlag, New York, 1996.
\end{thebibliography}
\end{document}